\documentclass[11pt,a4paper,twoside]{amsart}

\usepackage{amsmath,amssymb,amsthm}
\usepackage{inputenc}
\usepackage{tikz} 
\usepackage{float}
\floatstyle{boxed} 
\usepackage{graphicx,subfigure}
\usepackage{caption}
\usepackage{pdfsync} 
\usepackage[T1]{fontenc} 
\usepackage{enumitem}
\usepackage{mathtools}
\usepackage{mathrsfs}
\usepackage{xcolor}
\usepackage[colorlinks=true, linkcolor=blue, citecolor=red, urlcolor=magenta,hyperfootnotes=false]{hyperref}
\usepackage{cleveref}

\def\@begintheorem#1#2{\par\bgroup{\sc #1 \ #2. }  \it \\\ignorespace }
\def\@opargbegintheorem#1#2#3{\par\bgroup{\sc #1\ #2 \ (#3).}  \it  \ignorespace}
\def\@endtheorem{\egroup}

\theoremstyle{plain}
\newtheorem{theorem}{Theorem}[section]

\theoremstyle{definition}

\theoremstyle{remark}
\newtheorem{remark}[theorem]{Remark}

\theoremstyle{plain}

\theoremstyle{plain}
\newtheorem{lemma}[theorem]{Lemma}

\theoremstyle{plain}
\newtheorem{corollary}[theorem]{Corollary}

\theoremstyle{plain}

\numberwithin{equation}{section}

\newcommand{\C}{{\mathbb C}}

\newcommand{\D}{\mathcal{D}}

\newcommand{\W}{{\mathbf{W}}}

\newcommand{\R}{{\mathbb R}}

\newcommand{\I}{\mathbb{I}}
\newcommand{\vv}{\mathbf{v}}
\renewcommand{\u}{\mathbf{u}}
\newcommand{\V}{\mathbf{V}}

\renewcommand{\S}{{\mathbf{S}}}

\newcommand{\Nu}{{\rm{N}}}

\newcommand{\Rt}{{\R}^3}
\newcommand{\SL}{\S\cdot L}

\newcommand{\hx}{{\hat{x}}}
\renewcommand{\H}{\mathcal{H}}

\DeclareMathOperator{\Realpart}{Re}
\renewcommand{\Re}{\Realpart}

{\left\lbrace\begin{array}{@{}l@{}}}%
{\end{array}\right.}

\newcommand{\sgn}{\mathop{\textrm{sign}}}
\DeclarePairedDelimiter{\abs}{\lvert}{\rvert}

\DeclarePairedDelimiter{\seq}{\lbrace}{\rbrace}

\newcommand{\sign}{\mathop{\textrm{sign}}}

\newcommand{\ignora}[1]{}
\newcommand{\verde}[1]{}

\title[A Hardy-type inequality and some spectral characterizations for 
 Dirac-Coulomb]
{A Hardy-type inequality and some spectral characterizations for the
 Dirac-Coulomb operator}
\date{\today}
\author[B.~Cassano, F.~Pizzichillo, L.~Vega]{Biagio Cassano, Fabio
  Pizzichillo and Luis Vega}

\address{B.~Cassano, Department of Theoretical Physics, NPI, Academy of Sciences, 25068 \v{R}e\v{z} (Czechia)}
\email{cassano@ujf.cas.cz}
\address{F.~Pizzichillo, BCAM - Basque Center for Applied Mathematics,
Alameda de Mazarredo 14, 48009 Bilbao (Spain)}
\email{fpizzichillo@bcamath.org}
\address{L.~Vega, BCAM - Basque Center for Applied Mathematics,
Alameda de Mazarredo 14, 48009 Bilbao (Spain)}
\email{lvega@bcamath.org}

\usepackage{csquotes}

\begin{document}
\begin{abstract}
We prove a sharp Hardy-type inequality for the Dirac operator.
We exploit this inequality to obtain spectral properties of the Dirac
operator
perturbed with Hermitian matrix-valued potentials $\V$ of Coulomb type: 
we characterise its eigenvalues in terms of the Birman--Schwinger
principle and we bound its discrete spectrum from below, showing that
the \emph{ground-state energy} is reached 
if and only if $\V$ verifies some {rigidity} conditions.
In the particular case of an electrostatic potential, these imply
that $\V$ is the Coulomb potential.
\end{abstract}
\maketitle

\section{Introduction and main results}
Firstly formulated in \cite{hardy1920note},
the Hardy inequality can be stated as follows: for $d\geq 3$,
the following holds
\begin{equation}\label{eq:classical}
  \frac{(d-2)^2}{4}\int_{\R^d}\frac{|f|^2}{|x|^2}\,dx
  \leq
    \int_{\R^d}
  |\nabla
  f|^2\,dx,
  \quad\text{for}\ f\in C^\infty_c(\R^d).
\end{equation}
This inequality is sharp, in the sense that the constant in the left hand
side can not be increased, and 
there exists a sequence of approximate attainers.
We refer to \cite{kufner2006prehistory} for a historical review on
the topic.
The Hardy inequality is an uncertainty principle:
it states that a function cannot be concentrated around one point (the
origin) unless its momentum is big, and vice-versa if its momentum is
small then the function has to be spread in the space.
More in general, the Hardy inequality answers to the fundamental need in the mathematics of comparing
$L^2$--weighted norms of a function with the
norm of its derivative.
In this paper, we are interested in Hardy-type inequalities
for the Dirac operator: we  exploit them to show spectral
properties of the Dirac operator perturbed with potentials of Coulomb type.

The free Dirac operator in $\Rt$ is defined by 
\[
H_0:=-i\alpha\cdot\nabla+m\beta =-i \sum_{j=1}^3 \alpha_j \partial_j + m\beta, 
\]
where $m > 0$,
\begin{equation*}
  \beta:=
  \begin{pmatrix}
\mathbb{I}_2&0\\
0&-\mathbb{I}_2
\end{pmatrix},
\quad
\alpha:=(\alpha_1,\alpha_2,\alpha_3),
\quad
\alpha_j:=\begin{pmatrix}
0& {\sigma}_j\\
{\sigma}_j&0
\end{pmatrix}\quad \text{for}\ j=1,2,3,
\end{equation*}
and 
$\sigma_j$ are the \emph{Pauli matrices}
\[
\quad{\sigma}_1 =
\begin{pmatrix}
0 & 1\\
1 & 0
\end{pmatrix},\quad {\sigma}_2=
\begin{pmatrix}
0 & -i\\
i & 0
\end{pmatrix},
\quad{\sigma}_3=
\begin{pmatrix}
1 & 0\\
0 & -1
\end{pmatrix}.
\]
It is well known (see \cite{thaller}) that $H_0$
is self-adjoint on $H^1(\R^3)^4$ and essentially
self-adjoint on $C_c^{\infty}(\R^3)^4$, moreover
$\sigma(H_0)=\sigma_{ess}(H_0)=(-\infty,-m]\cup[m,+\infty)$.

In \cite{kato1951fundamental},
Kato considered a general matrix-valued potential $\V:\R^3 \to
\C^{4\times 4}$ such that $\V(x)$ is Hermitian for almost all $x \in
\R^3$ and
\[
  |\V(x)| 
  :=
  \sup_{u \in \C^4} \frac{|\V(x)u|}{\abs{u}}
  \leq
  \frac{a}{|x|}+b,
  \quad \text{ for a.a. }x \in \R^3,
\]
for some $a,b \in \R$.
Exploiting the Kato-Rellich perturbation theory, he showed that
if $a<1/2$ then $H:= H_0 + \V$ is self-adjoint on $H^1(\R^3)^4$ and essentially
self-adjoint on $C_c^{\infty}(\R^3)^4$ (for a proof see also \cite[Theorem V
5.10]{kato2013perturbation}). A fundamental ingredient of his proof is the Hardy inequality \eqref{eq:classical} when $d=3$.

In general, when $a > 1/2$ the operator $H$ is not essentially self-adjoint, as
shown by Arai in \cite{arai1983essential}, but the phenomena change
when the potential $\V$ has some particular structure.
For example, when $\V$ is the Coulomb potential
$\V_C(x):= \nu / \abs{x}\mathbb{I}_4$ and $\abs{\nu} < \sqrt{3}/2$,
the operator $H_0 + \V_C$ is
self-adjoint on $H^1(\R^3)^4$ and essentially
self-adjoint on $C_c^{\infty}(\R^3)^4$,
as shown in \cite{rellich1953eigenwerttheorie,
  weidmann1971oszillationsmethoden, gustafson1973some}.
These results suggest that the Hardy inequality \eqref{eq:classical}
is not optimal for the study of the self-adjointness of perturbed
Dirac operators, since it
does not catch its matrix nature:
convenient Hardy-type inequalities for the Dirac operator have to be
considered.
  Indeed, in \cite{schmincke1972essential} Schmincke considered
  a Coulomb-type potential such that
  \begin{equation}
    \label{eq:schminke}
      \V_S(x)=V_S(x) \mathbb{I}_4,
      \quad
      V_S : \R^3 \to \R,
      \quad
      \sup_{x \in \R^3} \abs{x}\abs{V_S(x)} < \frac{\sqrt{3}}{2},
  \end{equation}
and he showed that $H_0 +\V_S$ is essentially self-adjoint.
The basic idea in his proof is
to introduce a suitable \emph{intercalary operator} $T$
and to regard $\V_S- T$  as a perturbation of $H_0 + T$.
After a careful reading of his proof, one realises that
he proved and used the following Hardy-type inequality:
\[
  \int_{\R^3}
  \abs*{\left(-i \alpha \cdot \nabla
      + \frac{i \alpha \cdot \hx}{2\abs{x}}\right)\psi}^2
  \, dx
  \geq
  \int_{\R^3} \frac{\abs{\psi}^2}{\abs{x}^2}
  \, dx,
  \quad\text{for}
    \ \psi\in C^\infty_c(\Rt)^4,
\]
being $\hx:=x/|x|$.
In fact, thanks to \eqref{eq:schminke}, we have
\[
  \abs*{\left(\V_S(x) -
     \frac{i\alpha\cdot x}{2\abs{x}^2}\right)\psi}^2
     =
     \left(\abs{V_S(x)}^2 + \frac{1}{4\abs{x}^2}\right)\abs{\psi}^2 \leq
     \frac{\abs{\psi}^2}{\abs{x}}.
\]
The result of Schmincke is not an immediate application
of the Kato-Rellich theory, since the operator $ i \alpha \cdot \hx
/2\abs{x}$ is not symmetric: see \cite{schmincke1972essential} for more details.

For $|\nu|>\sqrt{3}/2$ the operator $H_0+\V_C$ is not essentially self-adjoint on $C^\infty_c(\Rt)^4$ and infinite self-adjoint extensions can be constructed. 
Among all, when $|\nu|<1$ there exists a unique self-adjoint extension $H_D$, characterized by the fact that 
\begin{equation}\label{eq:cond.distinguished}
  \D(H_D) \subset \D(r^{-1/2})^4 \quad \text{ or equivalently}
  \quad
  \D(H_D) \subseteq H^{1/2}(\Rt)^4,
\end{equation}
where 
\begin{equation}\label{eq:def.D(r^-1/2)}
\D(r^{-1/2})^4:=\seq{\psi\in L^2(\Rt)^4:|x|^{-1/2}\psi\in
    L^2(\Rt)^4},
\end{equation}
    see \cite{burnap1981dirac,
    gallonemichelangeli2017self, klaus1979characterization,
    nenciu1976self, schmincke1972distinguished,
    wust1975distinguished,gallonemichelangeli2017self,gallone2017discrete}.
Since $H_D$ is the unique self-adjoint extension verifying \eqref{eq:cond.distinguished}, it is called \emph{distinguished} because it is the most physically meaningful extension.
    
  In \cite{kato1983holomorphic} Kato constructed the distinguished self-adjoint extension
    in the general case that 
  $\V$ is a Hermitian matrix-valued potential such that
  \begin{equation}\label{eq:cond.V}
    \sup_{x\in\Rt} |x||\V(x)|=: \nu <1.
  \end{equation}
  To prove his result, Kato exploited the following $4-$spinor Hardy-type
  inequality, firstly conjectured by Nenciu in \cite{nenciu1976self}:
  \begin{equation}\label{eq:hardy.ADV}
    \int_{\Rt}\frac{|\psi|^2}{|x|}\,dx \leq \int_{\Rt}\abs*{(-i\alpha\cdot\nabla+m\beta\pm i)\psi}^2|x|\,dx, \quad\text{for}
    \ \psi\in C^\infty_c(\Rt)^4.
  \end{equation}
Finally in \cite{adv2013self,adv2018erratum}, by means of the
\emph{Kato-Nenciu} inequality \eqref{eq:hardy.ADV}, it is proved that
\begin{equation}\label{eq:def.dist}
  \D(H_{D})= \seq{\psi\in \D(r^{-1/2})^4: (H_0+\V)\psi\in L^2(\Rt)^4},
\end{equation}
being $\D(r^{-1/2})^4$ defined in \eqref{eq:def.D(r^-1/2)}.

In \cite{dolbeault2000eigenvalues}, Dolbeault, Esteban and
S\'er\'e proved the validity of a min-max formula
to determine the eigenvalues in the gap of
the essential spectrum of the  Dirac operator
perturbed with Coulomb-like potentials $\V$ such that
\begin{equation}
  \label{eq:des.potential}
  \V(x):= V(x) \I_4,
  \quad
 \lim_{|x|\to+\infty}
 |V(x)|=0,
 \quad
 -\frac{\nu}{|x|}-c_1\leq V\leq c_2:=\sup (V),
\end{equation}
with $\nu\in(0,1)$ and $c_1,c_2\geq 0$, $c_1+c_2-1<\sqrt{1-\nu^2}$.
As a consequence of their results, they 
proved the following Hardy-type inequality:
\begin{equation}\label{eq:hardy.des}
\int_{\R^3} \frac{\abs{\sigma\cdot \nabla \varphi}^2}{a + \frac{1}{\abs{x}}} + 
	\int_{\R^3} \left(a - \frac{1}{\abs{x}}\right)\abs{\varphi}^2  \geq 0,
	\quad \text{ for all } a>0, \varphi \in C^\infty_c(\R^3)^2,
\end{equation}
see also 
\cite{dolbeault2004analytical} for a later direct analytical proof.
Thanks to this inequality, 
in \cite{estebanloss}, Esteban and Loss considered a general
electrostatic potential $V:\Rt\to \R$ such that that for some constant $c(V)\in (-1,1)$, $\Gamma:=\sup(V)<1+c(V)$ and for every $\varphi\in C^\infty_c(\Rt,\C^2)$,
\begin{equation}\label{eq:diseq:EL}
\int_{\Rt}\left( \frac{|\sigma\cdot\nabla \varphi|^2}{1+c(V)-V}+\left(1+c(V)+V\right)|\varphi|^2\right) dx\geq 0,
\end{equation}
and, for $\V:= V \I_4$, they proved that the operator $H_0+\V$ is
self-adjoint on the appropriate domain.
In particular, they could treat potentials such that
\begin{equation}
  \label{eq:elsV}
  -\frac{\nu}{\abs{x}} \leq {V(x)} < 1 + \sqrt{1-\nu^2},
  \quad
  \text{ with } \nu \in (0,1],
\end{equation}
obtaining the distinguished extension in the case that $\nu<1$, and
giving a definition of distinguished extension in the critical case $\nu=1$.
The inequality \eqref{eq:hardy.des} was then used
 by Esteban, Lewin and S\'er\'e 
to study the spectrum of the Dirac operator perturbed with
these potentials of this kind: in \cite{els2017domains} 
they provided details on the domain of the distinguished extension and
they showed the validity of a min-max formula for the eigenvalues in the spectral gap.
In order to give properties on the spectrum of the Dirac operator
perturbed with a general Coulomb-type Hermitian matrix-valued potential,
in the following theorem we prove a  generalized version of
\eqref{eq:hardy.ADV}. In it we use the \emph{spin angular momentum operator}
$\S$ and the \emph{orbital angular momentum} $L$, whose definitions can be found in \eqref{eq:defn.spin}.
\begin{theorem}\label{thm:hardy}
Let $m>0$ and $a\in (-m,m)$. 
Let $\psi$ be a distribution such that 
\begin{equation}\label{eq:(H-a)psi.in.L^2|x|}
\int_{\Rt}|(-i\alpha\cdot\nabla+m\beta-a)\psi|^2|x|\,dx<+\infty.
\end{equation}
Then $\psi,(1+2\SL)\psi  \in L^2(|x|^{-1})^4$ and
\begin{equation}\label{eq:hardy.con.a}
\tfrac{m^2-a^2}{m^2}\int_{\Rt}\frac{|\psi|^2}{|x|}\,dx\leq
\tfrac{m^2-a^2}{m^2}\int_{\Rt}\frac{|(1+2\S\cdot L)\psi|^2}{|x|}\,dx
\leq\int_{\Rt}\abs*{\left(-i\alpha\cdot\nabla + m\beta-a\right)\psi}^2|x|\,dx.
\end{equation}
The inequalities are sharp, in the sense that the constants on the left
hand side can not be improved.
If $a\neq 0$, all the attainers of \eqref{eq:hardy.con.a} are given by
the elements of the two(complex)-parameter family
$\seq{\psi^a_C}_{C\in\C^2}$, with
\begin{equation}\label{eq:psia}
\psi^a_C:=
\begin{cases}
\begin{pmatrix}
\phi^{a}_C\\
\chi^a_C
\end{pmatrix}&\text{if}\ a> 0,\\
\\
\begin{pmatrix}
-\chi^{-a}_C\\
\phi^{-a}_C
\end{pmatrix}
&\text{if}\ a< 0,
\end{cases}
\quad\text{and}\quad
\begin{array}{l}
\phi^a_C(x)=C \frac{e^{\sqrt{m^2-a^2}|x|}}{|x|^{1-\frac{a}{m}}},\\
\\
\chi^a_C(x)=\sqrt{\frac{m-a}{m+a}}\left(i\sigma\cdot\hx\right)\phi^a_C.
\end{array}
\end{equation}
\end{theorem}
\begin{remark}
  In the case that $a=0$, the inequality \eqref{eq:hardy.con.a}
  is attained by the functions $\psi_C^a$ defined in
  \eqref{eq:psia}, setting $a=0$, in the
  sense that
  \[
    \lim_{\epsilon \to 0}
    \int_{\{\abs{x}>\epsilon\}}
    \left[\abs{x}\abs*{\left(-i\alpha\cdot\nabla + m\beta \right)\psi_C^0}^2
      -\frac{|\psi_C^0|^2}{|x|}\right] \,dx = 0.
  \]
\end{remark}
In the following we exploit \Cref{thm:hardy} to describe
 the discrete spectrum of the distinguished realization 
 $H_D$ defined in \eqref{eq:def.dist}, when \eqref{eq:cond.V} holds.
We refer to
\cite{kato1983holomorphic,adv2013self,adv2018erratum} for details
on its definition and properties. 

From \cite[Theorem 4.7]{thaller} we know that 
\[
\sigma_{ess}(H_D)=\sigma_{ess}(H_0)=(-\infty,-m]\cup[m,+\infty),
\]
and the discrete spectrum $\sigma_d(H_D)\subset(-m,m)$.

Thanks to \Cref{thm:hardy}, for $a\in (-m,m)$ 
\begin{equation}\label{eq:H0-a.inv}
(H_0-a)^{-1}:L^2(|x|)^4\to L^2(|x|^{-1})^4\qquad\text{is well-defined and bounded,}
\end{equation}
and so we immediately deduce that
\[
\u (H_0-a)^{-1} \vv: L^2(\Rt)^4\to L^2(\Rt)^4\qquad\text{is well-defined and bounded,}
\]
where
\begin{equation}\label{eq:def.u.v}
\u(x):=|x|^{1/2}\V(x)\quad\text{and}\quad \vv(x):=|x|^{-1/2}\I_4.
\end{equation}
Thanks to this, in the following theorem we characterize all the eigenvalues in $(-m,m)$ of the operator $H_D$ in terms of a \emph{Birman-Schwinger} principle.
\begin{theorem}[Birman--Schwinger principle]\label{thm:birman.schwinger}
Let $\V$ be a Hermitian matrix-valued potential that verifies
\eqref{eq:cond.V},
and let $\u,\vv$ be defined as in \eqref{eq:def.u.v}.
Let $H_D$ be the distinguished self-adjoint realization defined in \eqref{eq:def.dist}, and  let $a\in (-m,m)$. 
Then
\[
a\in\sigma_{d}(H_D) \iff -1\in \sigma_{d}(\u (H_0-a)^{-1} \vv).
\]
Moreover, the multiplicity of $a$ as an eigenvalue of $H_D$ coincides with the multiplicity of $-1$ as an eigenvalue of $\u (H_0-a)^{-1} \vv$.
\end{theorem}

The discrete spectrum of the distinguished self-adjoint realization of $H_\nu:=H_0-\frac{\nu}{|x|}$ with $0<\nu< 1$, is given by
\begin{equation}\label{eq:spectralCoul}
\sigma_d(H_\nu):=
\seq{a_1,a_2,\dots},\quad
m \sqrt{1-\nu^2}=a_1\leq a_2\leq\dots\leq a_n\leq\dots\leq m,
\quad
\lim_{n\to +\infty}a_n=m,
\end{equation}
see \cite{thaller} for more details.
It is easy to check that for any $C\in\C^2$: 
\begin{equation}\label{eq:attainer.eigenvector}
\left(H_0-\frac{\nu}{|x|}-a_1\right)\psi^{a_1}_C=0,
\end{equation}
being $\psi^{a_1}_C$ defined in \eqref{eq:psia}.
In other words, the attainers of \eqref{eq:hardy.con.a} are eigenvalues of the Dirac operator coupled with the Coulomb potential.
Moreover, it is easy to prove that $a\in\sigma_d(H_\nu)$ if and only if $-a\in\sigma_d(H_{-\nu})$.  So $\sigma_d(H_{-\nu})=\seq{-a_1,-a_2,\dots}$, and  for any $C\in\C^2$: 
\begin{equation}\label{eq:attainer.eigenvector.meno}
\left(H_0+\frac{\nu}{|x|}+a_1\right)\psi^{-a_1}_C=0,
\end{equation}

In \cite{dolbeault2000eigenvalues, dolbeault2004analytical}, it is
considered a radially symmetric {electrostatic} potential as in
\eqref{eq:des.potential}, with $\nu\in(0,1)$ and $c_1,c_2\geq 0$, $c_1+c_2-1<\sqrt{1-\nu^2}.$
It is proved that the discrete spectrum of the distinguished self-adjoint realization $H_D$ is given by
\begin{equation}\label{eq:spectralV}
\sigma_d(H_\D):=
\seq{a_1,a_2,\dots},\quad
m\sqrt{1-\nu^2}\leq a_1\leq a_2\leq\dots\leq a_n\leq\dots\leq m,
\quad
\lim_{n\to +\infty}a_n=m.
\end{equation}
Finally, in \cite{els2017domains}, this was proved when \eqref{eq:elsV} holds true.

From \eqref{eq:spectralCoul} and \eqref{eq:spectralV},  we have that 
the value $ m\sqrt{1-\nu^2}$ is the lower bound for the absolute value of the elements of discrete
spectrum
in presence of an electrostatic potential.
Such lower bound is reached in the case of the Coulomb potential, as shown in \eqref{eq:attainer.eigenvector} and \eqref{eq:attainer.eigenvector.meno}.
Exploiting \Cref{thm:hardy}, in the next theorem we show that this
phenomenon is more general, proving that $m\sqrt{1-\nu^2}$ is the lower bounds
for the absolute value of the elements of the discrete spectrum of $H_D$ for a general Hermitian
matrix-valued potential $\V$ satisfying
\eqref{eq:cond.V}. Moreover we can characterize all the
potentials $\V$ such that the values $\pm m\sqrt{1-\nu^2}$ are
eigenvalues for $H_D$.
\begin{theorem}\label{thm:rigidity}
  Let $\V$ be a Hermitian matrix valued potential
  that verifies \eqref{eq:cond.V},
  and let $H_D$ be the distinguished self-adjoint
  realization defined in \eqref{eq:def.dist}.
Let $a\in\sigma_d(H_D)$, let $\mu(a)$ be its multiplicity and let $\psi\in \D(H_D)$  be an associated eigenfunction. Then, the following hold:
\begin{enumerate}[label=(\roman*)]
\item
\label{item:rigidity.i}
 $\abs{a}\geq m\sqrt{1-\nu^2}$;
\item
\label{item:rigidity.ii}
 ${a} =\pm m\sqrt{1-\nu^2}$ if and only if  $\psi=\psi^a_C$ for some
 $C\in\C^2$, where $\psi^a_C$ is defined in \eqref{eq:psia};
 in this case, $\V\psi^a_C=\mp \frac{\nu}{|x|}\psi^a_C$ and
 $\mu(a)\leq 2$; 
\item
\label{item:rigidity.iii}
in the case that  $a= \pm m\sqrt{1-\nu^2}$, then $\mu(a)=2$ if and only if
\[
\V(x)=
\begin{cases}
-\frac{\nu}{|x|}\I_4+
\begin{pmatrix}
\Nu^2\sigma\cdot\hx\mathbf{W^+}(x)\sigma\cdot\hx&i\Nu\sigma\cdot\hx\mathbf{W^+}(x)\\
-i\Nu\mathbf{W^+}(x)\sigma\cdot\hx&\mathbf{W^+}(x)
\end{pmatrix}&\text{if}\ a>0,\\
\\
\frac{\nu}{|x|}\I_4+
\begin{pmatrix}
\mathbf{W^-}(x)&i\Nu\mathbf{W^-}(x)\sigma\cdot\hx\\
-i\Nu\sigma\cdot\hx\mathbf{W^-}(x)&
\Nu^2\sigma\cdot\hx\mathbf{W^-}(x)\sigma\cdot\hx
\end{pmatrix}&\text{if}\ a<0,
\end{cases}
\qquad
\text{for}\ x\neq 0,
\]
where $\Nu=\sqrt{\frac{1-\sqrt{1-\nu^2}}{1+\sqrt{1-\nu^2}}}$,
and $\mathbf{W^+}(x)$ and $\mathbf{W}^-(x)$ are $2\times 2$  Hermitian matrices whose eigenvalues are respectively $\seq{\lambda^+_j(x)}_{j=1,2}$ and $\seq{\lambda_j^-(x)}_{j=1,2}$, and they verify
\begin{equation}\label{eq:cond.autoval}
\begin{matrix}
  -\tfrac{\nu}{|x|}(1+\sqrt{1-\nu^2})\leq\lambda_j^-(x)
  \leq 0\leq\lambda_j^+(x)\leq\tfrac{\nu}{|x|}(1+\sqrt{1-\nu^2}),
\end{matrix}
\quad\text{for}\ j=1,2.
 \end{equation}
\end{enumerate}
\end{theorem}

\begin{remark}
From \emph{\ref{item:rigidity.i}} we directly have that $0\notin\sigma_d(H_D)$.
\end{remark}

An immediate consequence of \Cref{thm:rigidity} is the following corollary.
\begin{corollary}\label{cor:coulomb}
Let
$\V(x)=V(x)\I_4$, with 
$V:\R^3\to\R$, and such that \eqref{eq:cond.V} holds.
Let $a\in\sigma_d(H_D)$ and assume that $a=\pm m\sqrt{1-\nu^2}$.
Then $V(x)=\mp\frac{\nu}{|x|}$.
\end{corollary}

In this paper we are considering $\nu<1$
in \eqref{eq:cond.V}, since in the \emph{critical} case, namely when $\nu=1$ in \eqref{eq:cond.V}, a definition of distinguished extension is not available for a general Hermitian matrix-valued potential.
Indeed, in the particular case of the Coulomb potential $\V_C(x)=\nu/\abs{x}\I_4$,
when  $|\nu|\geq 1$ many self-adjoint extensions can be built:
for  $|\nu|> 1$ none appears to be {distinguished} in some suitable sense, see
\cite{hogreve2012overcritical, voronov2007dirac, xia1999contribution}.
For electrostatic potentials such that \eqref{eq:cond.V} holds with ${\nu}=1$, a definition of distinguished extension is implied
by the results of \cite{estebanloss},
and in \cite{els2017domains} it is shown that
this extension is the physically relevant one since it is the limit in the norm resolvent sense of
potentials where the singularity has been removed with a cut-off. 
For the operator $H_0+\V_{rad}$, where 
\[
\V_{rad}(x):= \frac{1}{|x|} \left( \nu \mathbb{I}_4 +\mu \beta -i
\lambda \alpha\cdot{\hx}\beta\right),\quad\text{for}\
\nu,\mu,\lambda \in \mathbb R,
\]
a complete description of all the
self-adjoint extensions 
is given in \cite{coulombbiagio, coulombbiagio2}.
Under some conditions on the size of the constants $\nu,\mu,\lambda$,
a distinguished extension is selected
by means of a Hardy-type inequality and a quadratic
form approach. Nevertheless, in the particular case of the \emph{anomalous magnetic potential}
$\V(x)=\pm i \alpha \cdot \hx \beta\abs{x}^{-1}$
such criteria do not appear to be strong enough to select any extension, see \cite[Remark 1.10]{coulombbiagio} and \cite[Remark 1.6]{coulombbiagio2}.

This paper is organised as follows: in \Cref{sec:hardy} we prove
\Cref{thm:hardy}, in \Cref{sec:birman} we prove \Cref{thm:birman.schwinger}
and we prove \Cref{thm:rigidity} in \Cref{sec:rigidity};
finally, in \Cref{sec:appendix} we recall the partial wave
decomposition and related properties.

\subsection*{Acknowledgements}
This research is supported by
ERCEA Advanced Grant 2014 669689 - HADE, by the MINECO project MTM2014-53850-P, by
Basque Government project IT-641-13 and also by the Basque Government through the BERC
2014-2017 program and by Spanish Ministry of Economy and Competitiveness MINECO: BCAM
Severo Ochoa excellence accreditation SEV-2013-0323. The first author
also acknowledges the Istituto Italiano di Alta Matematica ``F.~Severi'' and
 the Czech Science Foundation 
(GA\v{C}R) within the project 17-01706S.
\section{Hardy-type inequalities for the Dirac Operator}
\label{sec:hardy}

We use the following abstract result.
\begin{lemma}\label{lem:abstract.lem}
  Let $\mathcal{S}, \mathcal{A}$ be respectively a symmetric and an anti-symmetric
operator on a complex Hilbert space $\mathcal{H}$. Then the following holds:
\[
  2\Re \langle
   \mathcal{A} u, 
  \mathcal{S} u
  \rangle_{\H}
  =
  \langle [\mathcal{S},\mathcal{A}] u,u\rangle_{\H},
\]
where $[\mathcal{S},\mathcal{A}]:=\mathcal{SA}-\mathcal{AS}$ is the 
commutator of the operators $\mathcal{S}$ and $\mathcal{A}$.
\end{lemma}
\begin{proof}
  The proof is a simple computation:
  \[
      2\Re \langle \mathcal{A} u,
       \mathcal{S} u
       \rangle_\H
       =
      \langle \mathcal{A} u,
       \mathcal{S} u
       \rangle_\H
       +
      \overline{\langle \mathcal{A} u,
       \mathcal{S} u
       \rangle_\H}
     = \langle  \mathcal{S}\mathcal{A} u,
       u
       \rangle_\H
       -
       \langle  \mathcal{A}\mathcal{S} u,
       u
       \rangle_\H.
      \qedhere
  \]
\end{proof}

\begin{proof}[Proof of \Cref{thm:hardy}]
Let us firstly assume that $\psi \in \mathcal S(\Rt)^4$.
The proof descends immediately from the explicit computation of the following square:
\begin{equation}\label{eq:proof.hardy0}
\begin{split}
0&\leq \int_{\Rt}
  \abs{x}
\bigg\lvert(-i \alpha \cdot \nabla + m \beta - a)\psi
        - i \alpha \cdot \hx 
        \Big(1- \frac{a}{m} \beta \Big)(1+ 2\SL)\frac{\psi}{|x|}\bigg\rvert^2 \,dx
\\
&=
\int_{\Rt}
  \abs{x}
  \abs*{(-i \alpha \cdot \nabla + m \beta - a)\psi}^2 \,dx
-  \frac{m^2-a^2}{m^2} \int_{\R^3}
    \frac{\abs{(1+2\SL)\psi}^2}{\abs{x}}\,dx.
\end{split}
\end{equation}
Thanks to the fact that
\begin{equation}\label{eq:proof.hardy1}
\begin{split}
\int_{\Rt}\abs{x}
\bigg\lvert(-i \alpha \cdot \nabla + &m \beta - a)\psi
        - i \alpha \cdot \hx 
        \Big(1- \frac{a}{m} \beta \Big)(1+ 2\SL)\frac{\psi}{|x|}\bigg\rvert^2 \,dx
\\
  = &
  \int_{\R^3} \abs{x} \abs{(-i \alpha \cdot \nabla +m \beta - a)\psi}^2 \,dx
  +
  \int_{\R^3}|x|\abs*{\Big(1- \frac{a}{m}\beta\Big) 
    (1+2\SL)\frac{\psi}{\abs{x}}}^2 \,dx \\
  & -2\Re\int_{\R^3}
  (-i\alpha\cdot\nabla + m\beta - a)\psi
  \cdot
  \overline{i\alpha\cdot\hx \Big(1- \frac{a}{m}\beta\Big) (1+2\SL)\psi}
  \,dx,
\end{split}
\end{equation}
to prove \eqref{eq:proof.hardy0} it is enough to prove that
\begin{equation}\label{eq:proof.hardy.toprove}
\begin{split}
-2\Re\int_{\R^3}
  (-&i\alpha\cdot\nabla + m\beta - a)\psi
  \cdot
  \overline{i\alpha\cdot\hx \Big(1- \frac{a}{m}\beta\Big) (1+2\SL)\psi}
  \,dx  \\
&+ \int_{\R^3}|x|
 \bigg\lvert
 \Big(1- \frac{a}{m}\beta\Big) 
    (1+2\SL)\frac{\psi}{\abs{x}}
    \bigg\rvert^2 \,dx =
  -  \frac{m^2-a^2}{m^2} \int_{\R^3}
    \frac{\abs{(1+2\SL)\psi}^2}{\abs{x}}\,dx.
    \end{split}
\end{equation}

Let us firstly prove that
\begin{equation}\label{eq:doppioprod=II}
\begin{split}
 -2\Re\int_{\R^3}
  (-i\alpha\cdot\nabla + m\beta - a)\psi
  \cdot
  \overline{i\alpha\cdot\hx \Big(1- \frac{a}{m}\beta\Big) (1+2\SL)\psi}
  \,dx   \\
 =
 - 2\Re\int_{\R^3}(1+2\SL)\frac{\psi}{|x|}
  \cdot
  \overline{ \Big(1- \frac{a}{m}\beta\Big) (1+2\SL)\psi}
  \,dx.
 \end{split}
\end{equation}
With an explicit computation (see for example \cite[Equation 4.102]{thaller}) we get that
\begin{equation}
\label{eq:alpha.grad.polar}
  -i\alpha\cdot \nabla = 
  - i \alpha \cdot \hx \left (\left(\partial_r + \frac{1}{\abs{x}}\right)\I_4
    - \frac{1}{|x|}(1+ 2\SL)\right),
  \end{equation}
  and so the last term in \eqref{eq:proof.hardy1} can be expanded as follows:
\[
  \begin{split}
   -2\Re\int_{\R^3}
  (-i\alpha&\cdot\nabla + m\beta - a)\psi
  \cdot
  \overline{i\alpha\cdot\hx \Big(1- \frac{a}{m}\beta\Big) (1+2\SL)\psi}
  \,dx   
  \\
      =\,&
   2\Re\int_{\R^3}
  \Big(\partial_r + \frac{1}{\abs{x}}\Big) \psi
  \cdot
  \overline{ \Big(1- \frac{a}{m}\beta\Big) (1+2\SL)\psi}
  \,dx 
  \\
  & 
  - 2\Re\int_{\R^3}
   (1+2\SL)\frac{\psi}{\abs{x}}
  \cdot
  \overline{ \Big(1- \frac{a}{m}\beta\Big) (1+2\SL)\psi}
  \,dx 
  \\
  &  -2\Re\int_{\R^3}
  ( m\beta - a)\psi
  \cdot
  \overline{i\alpha\cdot\hx \Big(1- \frac{a}{m}\beta\Big)  (1+2\SL)\psi}
  \,dx
  \\
  =:& \,  I + II + III.
  \end{split}
\]
We show that $I=III=0$. 

Indeed,
the operator $\left(\partial_r + \frac{1}{|x|}\right)$ is skew-symmetric,
and the operator
$\Big(1+ \frac{a}{m}\beta\Big) (1+2\SL)$ is symmetric, since $\beta$ and $\SL$ are symmetric operators and they commute. 
Moreover,
\[
\Big[\partial_r + \frac{1}{\abs{x}}, \Big(1+ \frac{a}{m}\beta\Big) (1+2\SL)\Big]=0.
\]
So we can conclude that $I=0$, thanks to \Cref{lem:abstract.lem}.

Let us focus on $III$. Since $\beta^2=\I_4$, and $-i\alpha\cdot\hx$ and $\beta$ anti-commute, we rewrite
\[
  \begin{split}
    III = &-2\Re\int_{\R^3}
  ( m\beta - a)\psi
  \cdot
  \overline{i\alpha\cdot\hx \Big(1- \frac{a}{m}\beta\Big)  (1+2\SL)\psi}\,dx\\
  =& 
    \frac{1}{m} 2\Re\int_{\R^3}
    ( m\beta - a)\psi
    \cdot
    \overline{ (m\beta + a)   i\alpha\cdot\hx  \beta(1+2\SL)\psi}
    \,dx
    \\
    = &
    \frac{m^2 - a^2}{m} 2\Re\int_{\R^3}
    \psi \cdot
    \overline{ i\alpha\cdot\hx  \beta(1+2\SL)\psi}\,dx,
  \end{split}
\]
where, in the last equality, we used the fact that $(m\beta+a)$ is symmetric.
Since the operator $\beta(1+2\SL)$ is symmetric, the operator $i\alpha\cdot\hx$ is skew-symmetric and they anti-commute (see \cite[Equation 4.108]{thaller}), we have that the operator $ i\alpha\cdot\hx  \beta(1+2\SL)$ is skew-symmetric.
Finally, the identity operator is symmetric, and it trivially commutes with $i\alpha\cdot\hx  \beta(1+2\SL)$. Thus,
\Cref{lem:abstract.lem} let us conclude that $III=0$, and so \eqref{eq:doppioprod=II} is proved.

Finally, thanks to \eqref{eq:doppioprod=II} we get that
\[
\begin{split}
   \int_{\R^3} 
    |x|
    \Big\lvert&\Big(1- \frac{a}{m}\beta\Big) (1+2\SL)\frac{\psi}{|x|}\Big\rvert^2 \,dx
    - 2\Re\int_{\R^3}
  (1+2\SL)\frac{\psi}{\abs{x}}
  \cdot
  \overline{ \Big(1- \frac{a}{m}\beta\Big) (1+2\SL)\psi}
  \,dx 
     \\
&=-
     \Re \int_{\R^3}
    \Big(1 + \frac{a}{m}\beta\Big) (1+2\SL)\psi \cdot
    \overline{\Big(1- \frac{a}{m}\beta\Big)(1+2\SL) \frac{\psi}{\abs{x}}}\,dx
    \\
&=-  \frac{m^2-a^2}{m^2} \int_{\R^3}
    \frac{\abs{(1+2\SL)\psi}^2}{\abs{x}}\,dx.
  \end{split}
\]
and so \eqref{eq:proof.hardy.toprove} is proved. 
Finally we get \eqref{eq:hardy.con.a} combining \eqref{eq:proof.hardy0} and \eqref{eq:1<(1+SL)} .

Let us assume now that $\psi$ is a distribution verifying \eqref{eq:(H-a)psi.in.L^2|x|}.
Then, there exists a sequence $\seq{\varphi_n}_n\subset C^\infty_c(\Rt)^4$ such that 
\[
\varphi_n\to (H_0-a)\psi\quad \text{in}\ L^2(|x|)^4.
\]
The fundamental solution of $(H_0-a)$ is given by
\[
\phi^a (x):=\frac{e^{-\sqrt{m^2-a^2}|x|}}{4\pi|x|}\left(a+m\beta +\Big(1+\sqrt{m^2-a^2}|x|\Big)\,i\alpha\cdot\frac{x}{|x|^2}\right)\quad \text{for }x\in\Rt\setminus\{0\},
\]
Since $\phi^a$ has exponential decay at infinity, we get that $\psi_n:=\phi^a*\varphi_n\in \mathcal{S}(\Rt)^4$, so it verifies
\begin{equation}\label{eq:hardy.con.a.psi_n}
\int_{\Rt}\abs*{\left(H_0-a\right)\psi_n}^2|x|\,dx\geq
\tfrac{m^2-a^2}{m^2}\int_{\Rt}\frac{|(1+2\S\cdot L)\psi_n|^2}{|x|}\,dx\geq
\tfrac{m^2-a^2}{m^2}\int_{\Rt}\frac{|\psi_n|^2}{|x|}\,dx.
\end{equation}
By definition, $(H_0-a)\psi_n=\varphi_n$ and so we have that
\begin{equation}\label{eq:conv.(H-a)psi}
(H_0-a) \psi_n \to (H_0-a) \psi \quad \text{in}\ L^2(|x|)^4.
\end{equation}
Combining \eqref{eq:conv.(H-a)psi} and \eqref{eq:hardy.con.a.psi_n}, we deduce that both $\seq{(1+2\S\cdot L)\psi_n}_n$ and $\seq{\psi_n}_n$ are Cauchy sequences of $L^2(|x|^{-1})^4$. So, there exist $\eta,\theta\in L^2(|x|^{-1})^4$ such that
\begin{gather}
\label{eq:conv.eta}
\psi_n \to \eta \quad \text{in}\ L^2(|x|^{-1})^4,\\
\label{eq:conv.theta}
(1+2\S\cdot L)\psi_n\to \theta\quad \text{in}\ L^2(|x|^{-1})^4.
\end{gather}
Taking the limit on $n$ in \eqref{eq:hardy.con.a.psi_n}, we have that
\begin{equation}\label{eq:hardy.con.a.limit}
\int_{\Rt}\abs*{\left(H_0-a\right)\psi}^2|x|\,dx\geq
\tfrac{m^2-a^2}{m^2}\int_{\Rt}\frac{|\theta|^2}{|x|}\,dx\geq
\tfrac{m^2-a^2}{m^2}\int_{\Rt}\frac{|\eta|^2}{|x|}\,dx.
\end{equation}
Thanks to \eqref{eq:conv.eta} we deduce that, in the sense of distributions, the following hold
\begin{gather}
\label{eq:conv.(H-a)psi.dist}
(H_0-a)\psi_n \to (H_0-a)\eta,\\
\label{eq:conv.theta.dist}
(1+2\S\cdot L)\psi_n\to (1+2\S\cdot L)\eta.
\end{gather}
Thus, combining \eqref{eq:conv.(H-a)psi} with \eqref{eq:conv.(H-a)psi.dist} and \eqref{eq:conv.theta} with \eqref{eq:conv.theta.dist}, we have that
\begin{gather}
\label{eq:(H-a)psi=(H-a)eta}
(H_0-a) \psi= (H_0-a)\eta,\\
\label{eq:(1+SL)eta=theta}
(1+2\SL)\eta=\theta.
\end{gather}
Let us denote with $\langle \cdot , \cdot \rangle_{\D',\D}$ the usual pairing between a distribution and a test function. Then, for any $\varphi\in C^\infty_c(\Rt)^4$ we have that
\[
\langle \psi,\varphi\rangle_{\D',\D}=
\int_{\Rt} (H_0-a)\psi\cdot\overline{\phi^a*\varphi}\,dx=
\int_{\Rt} (H_0-a)\eta\cdot \overline{\phi^a*\varphi}\,dx=
\langle \eta,\varphi\rangle_{\D',\D},
\]
where we used \eqref{eq:(H-a)psi=(H-a)eta} in the third equality. 
For this reason, we can conclude that
\begin{equation}\label{eq:psi=eta}
\psi=\eta\in L^2(|x|^{-1})^4,
\end{equation}
and thanks to \eqref{eq:(1+SL)eta=theta}
\begin{equation}\label{eq:(1+2SL)psi=theta}
(1+2\SL)\psi=(1+2\SL)\eta=\theta.
\end{equation}
Finally, combining \eqref{eq:hardy.con.a.limit}, \eqref{eq:psi=eta}, and \eqref{eq:(1+2SL)psi=theta} we can conclude that $\psi$ verifies \eqref{eq:hardy.con.a}.
In particular, by a density argument, we get that $\psi$ verifies \eqref{eq:proof.hardy0}.

Let us finally assume that $\psi$ is an attainer of \eqref{eq:hardy.con.a}, that is
\begin{equation}\label{eq:hardy.con.a.attainer}
\int_{\Rt}\abs*{\left(-i\alpha\cdot\nabla + m\beta-a\right)\psi}^2|x|\,dx=
\tfrac{m^2-a^2}{m^2}\int_{\Rt}\frac{|(1+2\S\cdot L)\psi|^2}{|x|}\,dx=
\tfrac{m^2-a^2}{m^2}\int_{\Rt}\frac{|\psi|^2}{|x|}\,dx.
\end{equation}
We can decompose $\psi$ as in \eqref{eq:dec.armonic}, that is
\[
\psi(x)=
  \sum_{j,k_j,m_j} 
\frac{1}{r}\left(
f^+_{m_j,k_j}(r)\Phi^+_{m_j,k_j}(\hx)+
f^-_{m_j,k_j}(r)\Phi^-_{m_j,k_j}(\hx)\right).
\]
From the second equality of \eqref{eq:hardy.con.a.attainer}, and thanks to \eqref{eq:(1+SL)radial}, we directly have $f^\pm_{m_j,k_j}=0$ for $k_j\neq\pm 1$, or equivalently for $j\neq 1/2$. 
Let us focus on the first equality of \eqref{eq:hardy.con.a.attainer}. 
Thanks to \eqref{eq:proof.hardy0}, we get that
\[
  0 = 
  \int_{\R^3}
  \abs{x}
  \abs*{(-i \alpha \cdot \nabla + m \beta - a)\psi
        - i \alpha \cdot \hx 
        \Big(1- \frac{a}{m} \beta \Big)(1+ 2\SL)\frac{\psi}{\abs{x}}}^2 \,dx
\]
and so, 
\[
(-i \alpha \cdot \nabla + m \beta - a)\psi
        - i \alpha \cdot \hx 
        \Big(1- \frac{a}{m} \beta \Big)(1+ 2\SL)\frac{\psi}{\abs{x}}=0.
\]
Multiplying both therms by $i\alpha\cdot\hx$ and using \eqref{eq:alpha.grad.polar} we get that
\begin{equation}\label{eq:quadrato.no.rad}
\left(\partial_r+\frac{1}{|x|}+i\alpha\cdot\hx(m\beta-a)\right)\psi-\frac{a}{m}\beta(1+2\SL)\frac{\psi}{|x|}=0.
\end{equation}
The action of all the operators appearing in \eqref{eq:quadrato.no.rad} leaves invariant the decomposition in partial wave subspaces.
Thanks to \eqref{eq:(1+2SL)=-kBeta}, we get that for $m_{1/2}=\pm 1/2$ and $k_{1/2}=\pm 1$ we have
\begin{equation}\label{eq:sistem.attainer}
\begin{pmatrix}
\partial_r+\dfrac{ak_{1/2}}{m r} & -(m+a)\\
-(m-a)& \partial_r+\dfrac{ak_{1/2}}{m r}
\end{pmatrix}\cdot
\begin{pmatrix}
f^+_{m_{1/2},k_{1/2}}\\
f^-_{m_{1/2},k_{1/2}}
\end{pmatrix}=0.
\end{equation}
The only solution of \eqref{eq:sistem.attainer} that is integrable at $+\infty$ is
\begin{equation}
\begin{pmatrix}
f^+_{m_{1/2},k_{1/2}}\\
f^-_{m_{1/2},k_{1/2}}
\end{pmatrix}=
\begin{pmatrix}
e^{-\sqrt{m^2-a^2}r} r ^{-a k_{1/2}/m}\\
-\sqrt{\frac{m+a}{m-a}} e^{-\sqrt{m^2-a^2}r} r ^{-a k_{1/2}/m}
\end{pmatrix}.
\end{equation}
Then $\left(f^+_{m_{1/2},k_{1/2}},f^-_{m_{1/2},k_{1/2}}\right)\in
L^2(0,+\infty)^2$ if and only if $ak_{1/2}\leq0$. So if $a > 0$,we
have to assume $k_{1/2}=-1$, and if $a < 0$ we have $k_{1/2}=1$.
Remembering that 
\begin{align*}
\Phi^+_{\frac12,-1}&= 
  \frac1{\sqrt{4\pi}}
  \begin{pmatrix}
    i
    \\
    0
    \\
    0
    \\
    0
  \end{pmatrix}, 
  &
  \Phi^+_{-\frac12,-1}&= 
  \frac1{\sqrt{4\pi}}
  \begin{pmatrix}
    0
    \\
    i
    \\
    0
    \\
    0
  \end{pmatrix}, 
  \\
  \Phi^-_{\frac12,-1}&= 
  \frac1{\sqrt{4\pi}}
  \begin{pmatrix}
    0
    \\
    0
    \\
    \sigma \cdot \hx \cdot 
    \begin{pmatrix}
      0 \\ 1
    \end{pmatrix}
  \end{pmatrix}, 
  &
  \Phi^-_{-\frac12,-1}&= 
  \frac1{\sqrt{4\pi}}
  \begin{pmatrix}
    0
    \\
    0
    \\
    \sigma \cdot \hx \cdot 
    \begin{pmatrix}
      0 \\ 1
    \end{pmatrix}
  \end{pmatrix}, 
\end{align*}
  we conclude the proof.
\end{proof}

\section{Birman--Schwinger Principle for the Dirac--Coulomb operator}
\label{sec:birman}
This section is devoted to the proof of \Cref{thm:birman.schwinger}. 

Let $a\in(-m,m)$, $a\in\sigma_d(H_D)$ and
$\psi\in\D(H_D)\setminus\seq{0}$ such that $(H_0+\V-a)\psi=0$. We have
that, in the sense of distributions,  
\begin{equation}\label{eq:H-a=-V}
(H_0-a)\psi=-\V\psi. 
\end{equation}
 Since $\psi\in \D(H_D)\subset\D(r^{-1/2})^4$, then $f:=\u\psi\in L^2(\Rt)^4$ and $\vv f\in \D(r^{1/2})^4$. 
Thanks to \eqref{eq:H0-a.inv} we can apply $(H_0-a)^{-1}$ to
\eqref{eq:H-a=-V},
getting
$\psi=-(H_0-a)^{-1}\vv f$,
that implies
\begin{equation*}
  f=\u\psi=-\u (H_0-a)^{-1}\vv f.
\end{equation*}

Let now $-1$ be an eigenvalue of $\u(H_0-a)^{-1}\vv$ and let $f\in L^2(\Rt)^4$ be an eigenfunction.
Setting $\psi=(H_0-a)^{-1}\vv f$, we directly get that $\psi\in\D(r^{-1/2})$.
Reasoning as above, we get that $H_D\psi=a\psi$, and so $\psi\in H_D$ and $\psi$ is an eigenfunction of the eigenvalue $a$.

Finally, we point out that the shown procedure ensures that the multiplicity of $a$ as an eigenvalue of $H_D$ coincides with the multiplicity of $-1$ as an eigenvalue of $\u (H_0-a)^{-1} \vv$, and this concludes the proof.

\section{Proof of \texorpdfstring{\Cref{thm:rigidity}}{Theorem 1.4}}
\label{sec:rigidity}
This section is devoted to the proofs of  \Cref{thm:rigidity} and
\Cref{cor:coulomb}.

\begin{proof}[Proof of \Cref{thm:rigidity}]
Let $a \in (-m,m)$ and assume that there exists
$\psi\in\D(H_D)\setminus\seq{0}$ such that $(H_D-a)\psi=0$. Then, in
the sense of distributions, we get that $(H_0-a)\psi=-\V\psi$. Since $\psi\in\D(H_D)\subset\D(r^{-1/2})^4$ and thanks to the fact that $\V$ verifies \eqref{eq:cond.V}, we get that
\[
\int_{\Rt}|x||(H_0-a)\psi|^2\,dx=\int_{\Rt}|x||\V\psi|^2\,dx\leq \nu^2\int_{\Rt}\frac{|\psi|^2}{|x|}\,dx<+\infty .
\]
Thanks to \Cref{thm:hardy} we get that $\psi$ verifies \eqref{eq:hardy.con.a}, and so
\begin{equation}\label{eq:rigi}
\begin{split}
\nu^2\int_{\Rt} \frac{|\psi|^2}{|x|}&\geq\int_{\Rt}|x||(H_0-a)\psi|^2\,dx\geq
\tfrac{m^2-a^2}{m^2}\int_{\Rt}\frac{|(1+2\SL)\psi|^2}{|x|}\,dx\\
&\geq
\tfrac{m^2-a^2}{m^2}\int_{\Rt}\frac{|\psi|^2}{|x|}\,dx,
\end{split}
\end{equation}
So, $\nu^2\geq \tfrac{m^2-a^2}{m^2}$, that directly implies
\emph{\ref{item:rigidity.i}}.

Let us prove \emph{\ref{item:rigidity.ii}}. Let us assume that
$a^2=m^2(1-\nu^2)$. Then, from \eqref{eq:rigi} we deduce that $\psi$
is an attainer of \eqref{eq:hardy.con.a}: thanks to
\Cref{thm:hardy}, this is equivalent to say that there exists $C\in\C^2$ such that $\psi=\psi^a_C$, with $\psi^a_C$ defined in \eqref{eq:psia}.
This directly implies that $\mu(a)\leq 2$. 
Finally, thanks to \eqref{eq:attainer.eigenvector} and \eqref{eq:attainer.eigenvector.meno} we get that $0=\left(H_0- \sgn (a) \frac{\nu}{|x|}-a\right)\psi^a_C=(H_0+\V-a)\psi^a_C$, so
$\V\psi^a_C=-\sign(a)\frac{\nu}{|x|}\psi^a_C$.

Let us now prove \emph{\ref{item:rigidity.iii}}. 
We assume that $a$ is positive, that is $a=m\sqrt{1-\nu^2}$,
since the same approach can be used when $a$ is negative.
Moreover, let us assume that, for any $C\in \C^2$,
\begin{equation}
\label{eq:Vpsia=-nupsia}
\V\psi^a_{C}=-\frac{\nu}{|x|} \psi^a_{C}.
\end{equation}

Since
\begin{equation}\label{eq:def.Nu}
\sqrt{\frac{m-a}{m+a}}=
\sqrt{\frac{1-\sqrt{1-\nu^2}}{1+\sqrt{1-\nu^2}}}=\Nu,
\end{equation}
we have that 
\begin{equation}
\label{eq:def.psia}
\psi^a_{C}=
\frac{e^{-\sqrt{m^2-a^2}|x|}}{|x|^{1-a/m}}
\begin{pmatrix}
\mathbb{I}_2&0\\
0& i\Nu\sigma\cdot\hx
\end{pmatrix}
\cdot
\begin{pmatrix}
C\\
C
\end{pmatrix},
\end{equation}
where, with abuse of notation, we are denoting with
$\begin{pmatrix}
C\\C
\end{pmatrix}$ the $4$--component column vector.

Thanks to \eqref{eq:def.psia}, multiplying both therms of \eqref{eq:Vpsia=-nupsia} by $\left(\frac{e^{-\sqrt{m^2-a^2}|x|}}{|x|^{1-a/m}}\right)^{-1}$, we get that
\begin{equation}\label{eq:V.matrix}
\left(\V(x)+\frac{\nu}{|x|}\I_4\right)\cdot\begin{pmatrix}
\mathbb{I}_2&0\\
0& i\Nu\sigma\cdot\hx
\end{pmatrix}
\cdot
\begin{pmatrix}
C\\
C
\end{pmatrix}=0.
\end{equation}
Since both $\V$ and $\frac{\nu}{|x|}\I_4$ are Hermitian matrices, we can write
\begin{equation}\label{eq:dec.V+nu}
\left(\V+\frac{\nu}{|x|}\I_4\right)=:
\begin{pmatrix}
\W_{1,1}&\W_{1,2}\\
\W_{1,2}^*&\W_{2,2}
\end{pmatrix},
\end{equation}
where $\W_{1,1}$ and $\W_{2,2}$ are $2\times 2$ Hermitian matrices, $\W_{1,2}$ is a $2\times 2$ complex valued matrix and $\W_{1,2}^*$ is its adjoint matrix.

Combining \eqref{eq:V.matrix} and \eqref{eq:dec.V+nu}, we get that 
\begin{equation}\label{eq:system.V.vector}
\begin{cases}
(\W_{1,1}+i\Nu\W_{1,2}\,\sigma\cdot\hx)\cdot C=0,\\
(\W_{1,2}^*+i\Nu\W_{2,2}\,\sigma\cdot{\hx})\cdot C=0.
\end{cases}
\end{equation}
Since \eqref{eq:system.V.vector} holds for any $C\in\C^2$, we deduce that
\begin{align}
\label{eq:system.V.1}
\W_{1,1}+i\Nu\W_{1,2}\,\sigma\cdot\hx=0,\\
\label{eq:system.V.2}
\W_{1,2}^*+i\Nu\W_{2,2}\,\sigma\cdot{\hx}=0.
\end{align}
Taking the adjoint of  \eqref{eq:system.V.2}, and thanks to the fact that both $\sigma\cdot\hx$ and $\W_{2,2}$ are Hermitian matrices, we get that
\begin{equation}\label{eq:W_12}
\W_{1,2}=i\Nu\sigma\cdot\hx \W_{2,2}.
\end{equation}
Combining \eqref{eq:system.V.1} and \eqref{eq:W_12} we get that
\[
\W_{1,1}=\Nu^2\sigma\cdot\hx\W_{2,2}\sigma\cdot\hx.
\]
Setting for convenience $\W^+:=\W_{2,2}$, we can conclude that \eqref{eq:V.matrix} is equivalent to
\begin{equation}
\V(x):=-\frac{\nu}{|x|}\I_4
+
\begin{pmatrix}
\Nu^2\sigma\cdot\hx\mathbf{W}^+(x)\sigma\cdot\hx&i\Nu\sigma\cdot\hx\mathbf{W}^+(x)\\
-i\Nu\mathbf{W}^+(x)\sigma\cdot\hx&\mathbf{W}^+(x)
\end{pmatrix}.
\end{equation}

Finally, thanks to \eqref{eq:cond.V} we determine additional
properties on the matrix $\W^+(x)$. 
For any $x\in\Rt\setminus\seq{0}$, there exists $\seq{e_1(x),e_2(x)}$, an orthonormal basis of $\C^2$ of eigenvectors of $\W^+(x)$,
that is $\W^+(x) e_j(x)=\lambda_j^+(x) e_j(x)$, for $j=1,2$, with $\lambda_j^+(x)\in \R$.
Set
\[
u_j(x):=
\tfrac{1}{N^2+1}\begin{pmatrix}
e_j(x)\\
i\Nu\sigma\cdot\hx e_j(x)
\end{pmatrix}\quad\text{and}\quad
v_j(x):=
\tfrac{1}{N^2+1}
\begin{pmatrix}
i\Nu\sigma\cdot\hx e_j(x)\\
e_j(x)
\end{pmatrix},
\quad\text{for}\ j=1,2.
\]
The family $\seq{u_1(x),u_2(x),v_1(x),v_2(x)}$ is an orthonormal basis of $\C^4$.
Thus, $|x||\V(x)|\leq\nu$ if and only if $|x||\V(x) u_j(x)|\leq \nu$ and $|x||\V(x) v_j(x)|\leq \nu$ for $j=1,2$.

We have that, for $j=1,2$
\begin{equation}\label{eq:V.uj.vj}
\begin{split}
\V(x) u_j(x)&=-\frac{\nu}{|x|} u_j(x),\\
\V(x) v_j(x)&= \left(-\frac{\nu}{|x|}+\lambda_j^+(x)(N^2+1)\right) v_j(x).
\end{split}
\end{equation}
Since $|u_j(x)|=|v_j(x)|=1$, from \eqref{eq:V.uj.vj} we deduce that that $|x||\V(x)|\leq \nu$ if and only if
\begin{equation}\label{eq:cond.autoval.proof}
\left(-\nu+|x|\lambda_j^+(x)(N^2+1)\right)^2\leq\nu^2,\quad\text{for}\ j=1,2.
\end{equation}
From \eqref{eq:cond.autoval.proof} and \eqref{eq:def.Nu} we deduce
\eqref{eq:cond.autoval}, concluding the proof.
\end{proof}
\begin{proof}[Proof of \Cref{cor:coulomb}]
From \emph{\ref{item:rigidity.ii}} in \Cref{thm:rigidity}
we have that $V(x)\psi_C^a=\mp\frac{\nu}{|x|} \psi_C^a$ for some $C \in \C^2$,
and this implies the thesis. 
\end{proof}

\appendix
\section{Partial wave subspaces}\label{sec:appendix}
In this appendix, we recall the \emph{partial wave subspaces} associated to the Dirac equation. We sketch here this topic, referring to \cite[Section 4.6]{thaller} for further details.

Let $Y^l_n$ be the spherical harmonics. They are defined for $n = 0, 1, 2, \dots$, and $l =-n,-n + 1,\dots , n,$ and they satisfy $\Delta_{\mathbb{S}^2} Y^l_n= n(n + 1)Y^l_n$, where $\Delta_{\mathbb{S}^2}$ denotes the usual spherical Laplacian. Moreover, $Y^l_n$ form a complete orthonormal set in $L^2(\mathbb{S}^2)$.
For $j = 1/2, 3/2, 5/2, \dots , $ and $m_j = -j,-j + 1, \dots , j$, set
\[
\begin{split}
\psi^{m_j}_{j-1/2}&:=
\frac{1}{\sqrt{2j}}
\left(\begin{array}{c}
\sqrt{j+m_j}\,Y^{m_j-1/2}_{j-1/2}\\
\sqrt{j-m_j}\,Y^{m_j+1/2}_{j-1/2}\\
\end{array}\right),
\\
\psi^{m_j}_{j+1/2}&:=\frac{1}{\sqrt{2j+2}}
\left(\begin{array}{c}
\sqrt{j+1-m_j}\,Y^{m_j-1/2}_{j+1/2}\\
-\sqrt{j+1+m_j}\,Y^{m_j+1/2}_{j+1/2}\\
\end{array}\right);
\end{split}
\]
then  $\psi^{m_j}_{j\pm1/2}$ form a complete orthonormal set in $L^2(\mathbb{S}^2)^2$.
For $k_j:=\pm(j+1/2)$ we set
\[
\Phi^+_{m_j,\pm(j+1/2)}:=
\left(\begin{array}{c}
i\,\psi^{m_j}_{j\pm1/2}\\
0
\end{array}\right),
\quad
\Phi^-_{m_j,\pm(j+1/2)}:=
\left(\begin{array}{c}
0\\
\psi^{m_j}_{j\mp1/2}
\end{array}\right).
\]
Then, the set $\seq{\Phi^+_{m_j,k_j},\Phi^-_{m_j,k_j}}_{j,k_j,m_j}$ is a 
complete orthonormal basis of $L^2(\mathbb{S}^2)^4$ and
\begin{equation}\label{eq:(1+2SL)=-kBeta}
  (1+2\SL)\Phi_{m_j,k_j}=-k_j\beta\Phi_{m_j,k_j},
\end{equation}
where the \emph{spin angular momentum operator} $\S$ and the \emph{orbital angular momentum} $L$ are defined as
\begin{equation}\label{eq:defn.spin}
\mathbf{S}=
\frac{1}{2}\left(
\begin{array}{cc}
\sigma & 0\\
0 & \sigma
\end{array}
\right)\quad
\text{and}\quad
L:=-ix\wedge \nabla.
\end{equation}
\verde{We define the following space:
\begin{equation}
\mathcal{H}_{m_j,k_j}:=
\seq*{ 
\frac{1}{r}\left(
f^+_{m_j,k_j}(r)\Phi^+_{m_j,k_j}(\hx)+
f^-_{m_j,k_j}(r)\Phi^-_{m_j,k_j}(\hx)\right)
\in L^2(\Rt)
\mid
f^\pm_{m_j,k_j}\in L^2(0,+\infty)}.
\end{equation}}
So, we can write
\begin{equation}\label{eq:dec.armonic}
\psi(x)=
  \sum_{j,k_j,m_j} 
\frac{1}{r}\left(
f^+_{m_j,k_j}(r)\Phi^+_{m_j,k_j}(\hx)+
f^-_{m_j,k_j}(r)\Phi^-_{m_j,k_j}(\hx)\right)
\end{equation}
and, by definition,
\[
\int_{\Rt}|\psi|^2\,dx=  
\sum_{j,k_j,m_j} 
\int_0^{+\infty}|f^+_{m_j,k_j}(r)|^2+
|f^-_{m_j,k_j}(r)|^2\,dr.
\]
Thanks to \cite[Equation 4.109]{thaller} and \eqref{eq:dec.armonic}, we have that
\begin{equation}\label{eq:(1+SL)radial}
\begin{split}
\int_{\Rt}
\frac{|\psi|^2}{|x|}\,dx
&=
  \sum_{j,k_j,m_j} 
\int_0^{+\infty}
\frac{1}{r}
\left(
|f^+_{m_j,k_j}(r)|^2+
|f^-_{m_j,k_j}(r)|^2\right)dr,
\\
\int_{\Rt}
\frac{|(1+2\SL)\psi|^2}{|x|}\,dx
&=
  \sum_{j,k_j,m_j} 
\int_0^{+\infty}
\frac{k_j^2}{r}
\left(
|f^+_{m_j,k_j}(r)|^2+
|f^-_{m_j,k_j}(r)|^2\right)dr.
\end{split}
\end{equation}
From \eqref{eq:(1+SL)radial}, we directly deduce that
\begin{equation}\label{eq:1<(1+SL)}
\int_{\Rt}
\frac{|\psi|^2}{|x|}\,dx
\leq
\int_{\Rt}
\frac{|(1+2\SL)\psi|^2}{|x|}\,dx,
\end{equation}
and that \eqref{eq:1<(1+SL)} is attained if and only if
$f_{m_j,k_j}^\pm=0$  for $k_j\neq\pm 1$, or equivalently $j\neq 1/2$.

\end{document}